\documentclass[12pt]{amsart}

\usepackage{amsmath}
\usepackage{amsthm}
\usepackage{amssymb}
\usepackage{amsfonts}
\usepackage{bm}
\usepackage[shortlabels]{enumitem}
\usepackage[bookmarks=true,hyperindex,pdftex,colorlinks,citecolor=red, linkcolor=blue]{hyperref}
\usepackage{centernot} 
\usepackage{marginnote} 
\usepackage[left=3cm,right=3cm,top=3cm,bottom=2.5cm]{geometry}
\usepackage{bbm}
\usepackage[all]{xy}
\usepackage{tikz}
\usepackage{comment}
\usetikzlibrary{arrows}
\usetikzlibrary{shapes,decorations}
\usetikzlibrary{positioning}
\usepackage{xcolor}
\usepackage{bigints}
\usepackage{enumitem}

\theoremstyle{plain}
\newtheorem{thm}{Theorem}[section]
\newtheorem{prop}[thm]{Proposition}
\newtheorem{lemma}[thm]{Lemma}

\newtheorem{remark}[thm]{Remark}

\theoremstyle{definition}
\newtheorem{df}[thm]{Definition}
\newtheorem{q}[thm]{Question}

\newcommand{\N}{\mathbb{N}}

\newcommand{\norm}[1]{\left\Vert #1\right\Vert}

\definecolor{darkgreen}{rgb}{.2, .6, .2}

\newcommand{\KB}[0]{Kreiss bounded }
\newcommand{\UKB}[0]{uniformly Kreiss bounded }
\newcommand{\CB}[0]{Ces\`aro bounded }
\newcommand{\SCB}[0]{strongly Ces\`aro bounded }

\newcommand{\Tt}[0]{$(T_t)_{t\ge0}$}

\begin{document}
	
	\title[Growth rate of eventually positive Kreiss bounded $C_0$-semigroups on $L^p$ and $\mathcal{C}(K)$]{Growth rate of eventually positive Kreiss bounded $C_0$-semigroups on $L^p$ and $\mathcal{C}(K)$}
	
	\author[L. Arnold]{Loris Arnold}
	
	\author[C. Coine]{Cl\'ement Coine}
	
	\address[C. Coine]{Normandie Univ, UNICAEN, CNRS, LMNO, 14000 Caen, France}
	\email{clement.coine@unicaen.fr}
	
	\address[L. Arnold]{Jana I Jedrzeja \'Sniadeckich 8, 00-656 Warszawa, Poland}
	\email{larnold@impan.pl}

	\date{}

	\begin{abstract}
In this paper, we compare several Ces\`aro and Kreiss type boundedness conditions for a $C_0$-semigroup on a Banach space and we show that those conditions are all equivalent for a positive semigroup on a Banach lattice. Furthermore, we give an estimate of the growth rate of a Kreiss bounded and eventually positive $C_0$-semigroup $(T_t)_{t\ge 0}$ on certain Banach lattices $X$. We prove that if $X$ is an $L^p$-space, $1<p<+\infty$, then $\|T_t\| = \mathcal{O}\left(t/\log(t)^{\max(1/p,1/p')}\right)$ and if $X$ is an $(\text{AL})$ or $(\text{AM})$-space, then $\|T_t\|=\mathcal{O}(t^{1-\epsilon})$ for some $\epsilon \in (0,1)$, improving previous estimates.
	\end{abstract}

	

	
	\maketitle

	\section{Introduction}

	Let $(T_t)_{t\ge 0}$ be a $C_0$-semigroup on a complex Banach space $X$ and denote by $-A$ its generator. We say that $(T_t)_{t\ge 0}$ is Kreiss bounded if $\mathbb{C}_- \subseteq \rho(A)$ (the resolvent set of $A$) and for every $\lambda \in \mathbb{C}_+$,
	\[
	\norm{R(-\lambda,A)} \leq \frac{C}{\text{Re}(\lambda)},
	\]
	where, as usual, $\mathbb{C}_+ = \left\lbrace \lambda \in \mathbb{C}, \ \text{Re}(\lambda) > 0 \right\rbrace$ and $R(\mu,A)=(\mu-A)^{-1}$ is the resolvent of $A$ at $\mu\in \rho(A)$.
	Recall that if $X$ is a Banach lattice, we say that $(T_t)_{t\geq 0}$ is positive if, for every $t\geq 0$, $T_t$ is a positive operator, that is
	$$
	\forall t\geq 0, \ 0\leq f \in X \implies 0\leq T_tf.
	$$
	More generally, we say that $(T_t)_{t\ge 0}$ is individually eventually positive if for every $x\in X_+$ (the positive cone of $X$), there exists $t_x \geq 0$ such that, for every $t\geq t_x$, $T_tx$ is positive. We say that $(T_t)_{t\ge 0}$ is uniformly eventually positive if there exists $t_0 \geq 0$ such that for every $t\geq t_0$, $T_t$ is positive. There are examples of concrete differential equations associated to operators, such as certain Dirichlet-to-Neumann operators, which generate uniformly eventually positive but not positive $C_0$-semigroups. See \cite{Daners} and \cite{Danersgluck} for examples and complementary informations.

In this paper, a Banach lattice is the complexification of a real Banach lattice. We denote by $X_{\mathbb{R}}$ its real part, so that, by abusing the notation, $X = X_{\mathbb{R}} + i X_{\mathbb{R}}$. We refer to \cite[Section C-I 7.]{Ardentbook} and \cite[Section 2.2]{Meyer-N} for more details. In such space, the positive cone $X_+$ spans the entire space $X$. Examples of Banach lattices are $L^p(\Omega,\mu), 1\leq p\leq +\infty$ for some measure space $(\Omega,\mathcal{F},\mu)$ and the space $\mathcal{C}(K)$ where $K$ is a compact space.
In fact, the space $L^1$ on one hand, and $\mathcal{C}(K)$ on the other hand, are, respectively, examples of Banach lattices called $(\text{AL})$ and $(\text{AM})$-spaces. A Banach lattice $X$ is an $(\text{AL})$-space if the norm on $X$ satisfies
$$
\forall x,y \in X_+, \ \| x+y \| = \|x\| + \|y\|,
$$
that is, the norm is additive on the positive cone of $X$. A Banach lattice $X$ is an $(\text{AM})$-space if the norm on $X$ satisfies
$$
\forall x,y \in X_+, \ \| \sup(x,y) \| = \sup(\|x\|,\|y\|).
$$
If $X$ is a real $(\text{AM})$-space, then, by \cite[Proposition 1.4.7]{Meyer-N}, $X^*$ is an $(\text{AL})$-space. Moreover, by \cite[Section 2.2]{Meyer-N}, the dual of the complexification of $X$ can be identified with the complexification of $X^*$. In particular, if $X$ is a complex $(\text{AM})$-space, its dual space is an $(\text{AL})$-space. We will use this fact in Section \ref{Section4}.


	In this paper, we give an estimate for the growth rate of $\|T_t\|$ when $(T_t)_{t\ge 0}$ is Kreiss bounded and eventually positive on a Banach lattice $X$ where $X$ is either an $L^p$-space, an $(\text{AL})$-space or an $(\text{AM})$-space. Note that the Kreiss boundedness condition corresponds to the first order estimate in the Hille-Yosida theorem which characterizes bounded semigroups. When $X$ is finite dimensional, a Kreiss bounded semigroup is bounded but this is not the case if $X$ is infinite dimensional, even when $X$ is a Hilbert space, see \cite{EisnerZwart2}.
	
	Denote by $\omega_0(T)$ the growth rate of $T=(T_t)_{t\ge 0}$ that is, $\omega_0(T)$ is the infimum over all $\omega \in \mathbb{R}$ for which there exists $M\geq 1$ such that, for every $t\geq 0, \ \|T_t\| \leq Me^{\omega t}$. If $B$ is a closed operator on $X$, we define its spectral bound $s(B)$ by
	$$
	s(B) = \sup \left\lbrace \text{Re}(\lambda) \ : \ \lambda \in \sigma(B) \right\rbrace.
	$$
	It is well known that we have
	$$
	-\infty \leq s(-A) \leq \omega_0(T) < +\infty
	$$
	and that the inequality $s(-A) \leq \omega_0(T)$ can be strict, see \cite[Example 4.2]{Greiner}. However, for a uniformly eventually positive $C_0$-semigroup on $L^p, \, 1\leq p <+\infty$, we have equality, see \cite{Weis95}, \cite{Weis98} and \cite{Vogt}. We also have equality if $(T_t)_{t\ge 0}$ is individually eventually positive on an $(\text{AM})$-space, see \cite[Theorem 4]{AroGlu}. For a Kreiss bounded semigroup, we have, by definition, that $\sigma(A) \subset \overline{\mathbb{C}_+}$ so that $s(-A) \leq 0$. Hence, if $(T_t)_{t\ge 0}$ is positive on $L^p$ or on an $(\text{AM})$-space, then $\omega_0(T) \leq 0$ which implies that $(T_t)_{t\geq 0}$ cannot have an exponential growth.
	
	Our main results are the following. First, we prove that if $(T_t)_{t\ge 0}$ is a Kreiss bounded and uniformly eventually positive $C_0$-semigroup on $L^p, \, 1 < p <+\infty$, then $\norm{T_t}= \mathcal{O}\left( \dfrac{t}{\log(t)^{\max\{1/p,1/p'\}}} \right).$ Second, if $(T_t)_{t\ge 0}$ is a Kreiss bounded and individually (respectively uniformly) eventually positive on an $(\text{AL})$-space (resp. on an $(\text{AM})$-space),
we prove the stronger estimate $\norm{T_t}=\mathcal{O}(t^{1-\epsilon})$ for some $\epsilon \in (0,1)$.
	This result improves earlier estimates such as in \cite{RozVer}, where the authors prove that if $(T_t)_{t\ge 0}$ is Kreiss and positive on $X=L^p$ or $C_b(\Omega)$, $\|T_t\| = \mathcal{O}(t)$. Note that it is important for $X$ to be any of the spaces defined above. For example, in \cite[Example 3.1]{EisnerZwart2}, the authors construct a positive Kreiss bounded semigroup with exponential growth. However, when $X$ is a Hilbert space, it has been proven recently by the first author of this paper that a Kreiss bounded semigroup $(T_t)_{t\ge 0}$ (not necessarily positive) on $X$  satisfies the estimate $\norm{T_t}= \mathcal{O}\left( t/\sqrt{\log(t)} \right)$, see \cite{Loris}.
	
The paper is organized as follows. In Section $2$., we prove that an individually eventually positive $C_0$-semigroup $(T_t)_{t\ge 0}$ is Kreiss bounded if and only if it is Ces\`aro bounded, that is, for every $x\in X$ and every $t>0$,
		\[
		\frac{1}{t}\norm{\int_{0}^t T_s x ds} \leq C\|x\|.
		\]
This allows to verify any of the boundedness properties defined above using either the resolvent or the semigroup. We also prove that they are equivalent to stronger conditions, namely, the uniform Kreiss boundedness and the strong Ces\`aro boundedness properties, and equivalent to a weaker condition, the Abel boundedness property.

Next, in a third section, we prove the estimate for the growth rate of a uniformly eventually positive Kreiss bounded semigroup on $L^p$. In order to do so, we first obtain an auxiliary estimate in Proposition \ref{propAlmpACB} which can be applied to the semigroup and the adjoint semigroup. Those inequalities together will allow us to get the desired estimate for the norm of $T_t$ to obtain our first main result, Theorem \ref{maintheorem}. We then discuss the discrete case, namely, we obtain a bound for the growth rate of the powers of a single positive operator, see Theorem \ref{maintheoremdiscrete}.

Finally, in a fourth section, we first prove the estimate for the growth rate of an individually eventually positive Kreiss bounded semigroup on an $(\text{AL})$-space using simple estimates and deduce a similar estimate for uniformly eventually positive Kreiss bounded semigroups on $(\text{AM})$-spaces, using a duality argument.



	\section{Kreiss and Ces\`aro conditions for $C_0$-semigroups}\label{Section2}
	
	The following definitions are well known and studied in the discrete case, that is, when working with the powers of a single operator, see for instance \cite{Cuny1} and the references therein.
	We start by giving the appropriate definitions of Kreiss and Ces\`aro boundedness properties in the setting of $C_0$-semigroups.

	\begin{df}\label{defNotions}
		Let $(T_t)_{t\ge 0}$ be a $C_0$-semigroup on a Banach space $X$ and let $-A$ be its generator. We say that :
		\begin{enumerate}
			\item  $(T_t)_{t\ge 0}$ is Kreiss bounded if $\mathbb{C}_- \subseteq \rho(A)$ and for every $\lambda \in \mathbb{C}_+$,
			\[
			\norm{R(-\lambda,A)} \leq \frac{C}{\text{Re}(\lambda)}.
			\]
			\item  $(T_t)_{t\geq 0}$ is Abel bounded if $\mathbb{R}_- \subseteq \rho(A)$ and for every $\mu \in \mathbb{R}_+$,
            \[
             \norm{R(-\mu,A)} \leq \frac{C}{\mu}.
            \]
			\item $(T_t)_{t\ge 0}$ is \UKB if, for every $\alpha > 0, \beta\in \mathbb{R}$ and every $t>0$,
			\[
			\norm{\int_{0}^{t} e^{i \beta s }e^{-\alpha s}T_{s}xds } \leq \frac{C\norm{x}}{\alpha}.
			\]
			\item $(T_t)_{t\ge 0}$ is \CB if,  for every $x\in X$ and every $t>0$,
			\[
			\frac{1}{t}\norm{\int_{0}^t T_s x ds} \leq C\|x\|.
			\]
			\item $(T_t)_{t\ge 0}$ is \SCB if for every $x\in X, x^* \in X^*$ and every $t>0$,
			\[
			\frac{1}{t}\int_{0}^t |\langle T_sx ,x^* \rangle| ds \leq C\norm{x}\norm{x^*}.
			\]
		\end{enumerate}
		Here, $C$ denotes a constant which only depends on the semigroup.
	\end{df}
	
	
	\begin{prop}\label{propUKBimplKB}
		Let $(T_t)_{t\ge 0}$ be a $C_0$-semigroup on a Banach space $X$, with generator $-A$. If $(T_t)_{t\ge 0}$ is uniformly Kreiss bounded, then it is Kreiss bounded.
	\end{prop}

	\begin{proof}
		We know that $z \mapsto R(z,A)$ is analytic on $\mathcal{U} = \{ z\in \mathbb{C}, \ \text{Re}(z) <  -\omega_0(T) \}$ and that for every $z\in \mathcal{U}$ and every $x\in X$,
		$$
		R(z,A)x = - \int_0^{+\infty} e^{zs} T_sx ds.
		$$
		Define $f : z\mapsto \displaystyle \int_0^{+\infty} e^{zs} T_s ds$. Let us show that $f$ extends to an analytic function on $\mathcal{V} = \{ z\in \mathbb{C}, \ \text{Re}(z) < 0 \}$. Let $\mu>0$ and let us extend $f$ on $\{ \text{Re}(z) < -\mu \}$. Let $z = \alpha + i \beta$ with $\alpha < -\mu$. For $R>0$, integrating by parts gives
		\begin{align*}
			\int_0^R e^{zs}T_s ds
			& = \int_0^R e^{-\mu s} e^{(\alpha+\mu)s} e^{i\beta s} T_s ds \\
			& = \left[ e^{-\mu s} \left(\int_0^s e^{(\alpha+\mu)t} e^{i\beta t} T_t dt \right) \right]_0^R + \mu \int_0^R \left( e^{-\mu s} \int_0^s e^{(\alpha+\mu)t} e^{i\beta t} T_t dt \right) ds \\
			& = e^{-\mu R} \left(\int_0^R e^{(\alpha+\mu)t} e^{i\beta t} T_t dt \right) + \mu \int_0^R \left( e^{-\mu s} \int_0^s e^{(\alpha+\mu)t} e^{i\beta t} T_t dt \right) ds.
		\end{align*}
		The uniform Kreiss bounded property of $(T_t)_{t\ge 0}$ implies that the first term converges to $0$ as $R\to +\infty$ and that the integral from the second term will converge. Taking $R\to +\infty$ yields
		$$
		f(z) = \mu \int_0^{+\infty} \left( e^{-\mu s} \int_0^s e^{(\alpha+\mu)t} e^{i\beta t} T_t dt \right) ds.
		$$
		It is easy to check that this expression is analytic on $\{z=\alpha+i\beta, \ \text{Re}(z) < -\mu \}$. Since this holds for any $\mu>0$, we deduce that $f$ extends analytically on $\mathcal{V}$. Moreover, for every $z\in \mathcal{U}$
		$$
		-f(z)(zI - A) = -(zI - A)f(z) = I
		$$
		and this identity extends to $\mathcal{V}$, that is, for every $z\in \mathcal{V}$, $z\in \rho(A)$ and $f(z) = -R(z,A)$. Finally, for $x\in X$ and $\text{Re}(z) < -\mu$ we have the estimate
		\begin{align*}
			\|R(z, A)x\|
			& = \|f(z)x\| \\
			& \leq \mu \int_0^{+\infty} e^{-\mu s} \left\| \int_0^s e^{(\alpha+\mu)t} e^{i\beta t} T_t x dt \right\| ds \\
			& \leq  \left(\mu \int_0^{+\infty} e^{-\mu s} ds\right) \dfrac{C\norm{x}}{-\alpha-\mu} \\
			& = \dfrac{C\norm{x}}{-\alpha-\mu}.
		\end{align*}
		Letting $\mu\to 0$ yields $\|R(z,A)x\| \leq \dfrac{C\norm{x}}{-\text{Re}(z)}$, and this holds true for every $z\in \mathcal{V}$. This shows that $(T_t)_{t\ge 0}$ is Kreiss bounded.
	\end{proof}
	
	\begin{prop}\label{CNSUKB}
		Let $(T_t)_{t \geq 0}$ be a $C_0$-semigroup on a Banach space $X$. Then, the following conditions are equivalent:
		\begin{enumerate}
			\item $(T_t)_{t \geq 0}$ is uniformly Kreiss bounded.
			\item There exists $C>0$ such that for any $r \in \mathbb{R}$, any $t>0$ and any $x\in X$,
			\[
			\frac{1}{t}\norm{\int_{0}^{t} e^{irs}T_s xds} \leq C\norm{x}.
			\]
		\end{enumerate}
	\end{prop}
	
	\begin{proof}
		Let $x \in X$.
		
		$(1) \Rightarrow (2)$: if $(T_t)_{t \geq 0}$ is uniformly Kreiss then for each $r \in \mathbb{R}$,  $(e^{irt}T_t)_{t \geq 0}$ is also uniformly Kreiss with same constant, so it suffices to show $(2)$ with $r = 0$.
		Let $t>0$ and $\alpha > 0$. By integration by parts
		\begin{align*}
			\int_{0}^t T_sx ds &= \int_{0}^t e^{\alpha s}e^{-\alpha s}T_sx ds \\
			&= e^{\alpha t}\int_{0}^t e^{-\alpha s}T_sx ds - \int_0^t \alpha e^{\alpha s} \left(\int_0^s e^{-\alpha u}T_uxdu \right)ds.
		\end{align*}
		Hence
		\[
		\frac{1}{t}\norm{\int_{0}^{t} T_s xds}  \leq \frac{C e^{\alpha t}\norm{x}}{\alpha t} + \frac{C (e^{\alpha t}-1)\norm{x}}{\alpha t}.
		\]
		Finally, choosing $\alpha = \frac{1}{t}$ yields
		\[
		\frac{1}{t}\norm{\int_{0}^{t} T_s xds}  \leq C(2e -1)\norm{x}.
		\]
		
		$(2) \Rightarrow (1)$: let $\alpha > 0$ and $\beta \in \mathbb{R}$. Write $R_t = e^{i \beta t }T_t $.  The same computations as above give
		\[
		\int_{0}^t e^{-\alpha s}e^{i\beta s}T_sx ds = e^{-\alpha t}\int_{0}^t R_sx ds + \alpha \int_{0}^t e^{-\alpha s} \left( \int_0^s R_ux du \right) ds.
		\]
		Therefore, if  $(T_t)_{t \geq 0}$ satisifies $(2)$, we get
		\begin{align*}
			\norm{\int_{0}^t e^{-\alpha s}e^{i\beta s}T_sx ds}
			& \leq Cte^{-\alpha t} \|x\| + C \alpha  \|x\| \int_0^t e^{-\alpha s} s ds \\
			& = Cte^{-\alpha t} \|x\| + C \alpha  \|x\| \left(-\dfrac{te^{-\alpha t}}{\alpha} + \dfrac{1-e^{-\alpha t}}{\alpha^2} \right) \\
			& = \dfrac{C\|x\|}{\alpha} (1-e^{- \alpha t}) \\
			& \leq \dfrac{C\|x\|}{\alpha}.
		\end{align*}
		This concludes the proof.
	\end{proof}
	
	\begin{prop}\label{CNSCSB}
		Let $(T_t)_{t \geq 0}$ be a $C_0$-semigroup on a Banach space $X$. Then, $(T_t)_{t \geq 0}$ is \SCB if and only if there exists $C>0$ such that for each measurable function $\epsilon : \mathbb{R}_+ \rightarrow \mathbb{T}$, $(T_t)_{t\geq 0}$ satisfies
		\begin{equation}\label{CNSCSBfor}
			\norm{\frac{1}{t}\int_{0}^t \epsilon(s)T_sxds} \leq C\norm{x}, \quad \forall x \in X.
		\end{equation}
	\end{prop}
	
	\begin{proof}
		Assume $(T_t)_{t\geq 0}$ is strongly Ces\`aro bounded. we have
		\begin{align*}
			\norm{\frac{1}{t}\int_{0}^t \epsilon(s)T_sxds} &= \sup_{\norm{x^*}=1} \Big|\Big\langle \Big(\frac{1}{t}\int_{0}^t \epsilon(s)T_sds\Big) x, x^* \Big\rangle\Big| \\
			&= \sup_{\norm{x^*}=1}\Big| \frac{1}{t}\int_{0}^t \langle \epsilon(s)T_s x, x^* \rangle ds\Big|\\
			&\leq \sup_{\norm{x^*}=1}\frac{1}{t} \int_{0}^t \big|\langle T_s x, x^* \rangle\big| ds\\
			& \leq C\norm{x}.
		\end{align*}
		For the converse, we fix $x\in X$ and $x^*\in X^*$ and define $\epsilon : \mathbb{R}_+ \rightarrow \mathbb{T}$ by $\epsilon(s) =   \displaystyle\frac{\overline{\langle T_sx,x^* \rangle}}{\langle T_sx,x^* \rangle} $ when $\langle T_sx,x^* \rangle \neq 0$ and $\epsilon(s) =1$ when  $\langle T_sx,x^* \rangle = 0$.  The function $\epsilon$ is measurable and we have
		\begin{align*}
			\frac{1}{t} \int_{0}^t \big|\langle T_s x, x^* \rangle\big|ds
			& = \frac{1}{t} \int_{0}^t \langle \epsilon(s) T_s x, x^* \rangle ds &\\
			& = \Big\langle \frac{1}{t}\int_{0}^t \epsilon(s)T_sx ds, x^* \Big\rangle \\
			&\leq \norm{\frac{1}{t}\int_{0}^t \epsilon(s)T_sxds} \norm{x^*} \\
			& \leq C\norm{x}\norm{x^*}.
		\end{align*}
		This completes the proof.
	\end{proof}	
	
	\begin{remark} \label{rkSCBimplUKB}
		\begin{enumerate}
			\item Using Proposition \ref{CNSUKB} and \ref{CNSCSB}, it is easy to see that the strong Ces\`aro boundedness property implies the \UKB property $($it suffices to take $\epsilon(s) = e^{irs}$ in \eqref{CNSCSBfor}$)$.
			\item It is clear that if $(T_t)_{t\ge 0}$ is \SCB then it is Ces\`aro bounded.
		\end{enumerate}
	\end{remark}

	The following proposition shows that the four notions defined in Definition \ref{defNotions} are all equivalent for a uniformly eventually positive $C_0$-semigroup.
	
	\begin{prop}\label{propequiv}
		Let  $(T_t)_{t\ge 0}$ be an individually eventually positive $C_0$-semigroup on a Banach lattice $X$. Then the following assertions are all equivalent.
		\begin{enumerate}[label = (\roman*)]
			\item $(T_t)_{t\ge 0}$ is Ces\`aro bounded,
			\item $(T_t)_{t\ge 0}$ is strongly Ces\`aro bounded,
			\item $(T_t)_{t\ge 0}$ is uniformly Kreiss bounded,
			\item $(T_t)_{t\ge 0}$ is Kreiss bounded.
			\item[(v)] $(T_t)_{t\ge 0}$ is Abel bounded.
		\end{enumerate}	
	\end{prop}
	
	\begin{proof} $(i) \Rightarrow (ii)$ For $x \in X_+$, we let $t_x \geq 0$ be such that for every $t\geq t_x$, $T_tx$ is positive. Let $M_x = \sup_{u\in [0,t_x]} \ \|T_u\|$. Let $\epsilon : \mathbb{R}_+ \rightarrow \mathbb{T}$ be measurable. Since, for $t\geq t_x$, $T_tx$ is positive and $(T_t)_{t\ge 0}$ is \CB we have, for $t\geq t_x$,
		\begin{align*}
			\norm{\frac{1}{t}\int_{0}^t \epsilon(s)T_sxds}
			& = \norm{\frac{1}{t}\int_{0}^{t_x} \epsilon(s)T_sxds + \frac{1}{t}\int_{t_x}^{t} \epsilon(s)T_sxds} \\
			& \leq M_x\|x\| + \norm{\frac{1}{t}\Big|\int_{t_x}^t \epsilon(s)T_sxds \Big|} \\
			& \le M_x\|x\| + \norm{\frac{1}{t}\int_{t_x}^t |T_sx|ds } \\
			&= M_x\|x\| + \norm{\frac{1}{t}\int_{t_x}^t T_sxds} \\
			&= M_x\|x\| + \norm{\frac{1}{t}\int_{0}^t T_sxds - \frac{1}{t}\int_{0}^{t_x} T_sxds} \\
			& \le (2M_x+C)\norm{x}.
		\end{align*}
		We deduce that a similar estimate holds for a general $x\in X$.
		It follows from the uniform boundedness principle and Proposition \ref{CNSCSB} that $(T_t)_{t\ge 0}$ is strongly Ces\`aro bounded. \\
		$(ii) \Rightarrow (iii)$ follows from Remark \ref{rkSCBimplUKB} (1).\\
		$(iii) \Rightarrow (iv)$ follows from Proposition \ref{propUKBimplKB}.\\
		$(iv) \Rightarrow (v)$ follows from the definitions.\\
		$(v) \Rightarrow (iv)$ First, note that \cite[Theorem 1.2]{Ardentbook} holds true in the case of an individually eventually positive $C_0$-semigroup. Next, apply \cite[Corollary 1.3]{Ardentbook} and \cite[Corollary 1.4]{Ardentbook} to obtain $\mathbb{C}_- \subseteq \rho(A)$ as well as the desired estimate for the resolvent of the generator of $(T_t)_{t\ge 0}$.\\
		$(iv) \Rightarrow (i)$ Let $x\in X_+$. By \cite[Definition 2.8]{Ardentbook}, we have $\displaystyle \sup_{t>0} \frac{1}{t}\norm{\int_{0}^t T_s x ds} < +\infty.$
        We may now apply the uniform boundedness principle and this yields $(i)$.
	\end{proof}

	\begin{remark}
		If the semigroup is not individually eventually positive positive, this result does not hold true. For instance, for $A = -\begin{pmatrix}
		i & 1 \\
		0 & i
		\end{pmatrix},$
		it is easy to check that $$e^{-tA} = \begin{pmatrix}
		e^{it} & te^{it} \\
		0 & e^{it}
		\end{pmatrix}$$
		so that $\|e^{-tA}\| \sim t$ and $\left\| \displaystyle \int_0^t e^{-sA} ds \right\| \leq Ct$. Hence, the semigroup $\left( e^{-tA} \right)_{t\geq 0}$ is Ces\`aro bounded but it is not Kreiss bounded, since Kreiss boundedness is equivalent to the boundedness of the semigroup in the finite dimensional case.
	\end{remark}

	
	\section{Growth rate of uniformly eventually positive Kreiss bounded $C_0$-semigroups on $L^p$-spaces}\label{Section3}

Let $(\Omega, \mathcal{F}, \mu)$ be a $\sigma$-finite measure space and let $L^p$ be the Lebesgue space $L^p(\Omega, \mu)$. In this section, we study the growth rate of $\|T_t\|$ when $(T_t)_{t\ge 0}$ is a uniformly eventually positive $C_0$-semigroup on $L^p$. We recall that when $1<p<+\infty$, $L^p$ is reflexive so it follows from \cite[Page 9, 1.13]{Engel} that $(T_t^*)_{t\ge 0}$ is a $C_0$-semigroup on $L^{p'}$. Moreover, this semigroup is uniformly eventually positive. The main result of this section is the following.
	
	\begin{thm}\label{maintheorem}
Let $1 < p<\infty$ and let $(T_t)_{t\ge 0}$ be a Kreiss bounded and uniformly eventually positive $C_0$-semigroup on $L^p(\Omega)$.
Then, as $t\to +\infty$,
			$$\norm{T_t}= \mathcal{O}\left( \dfrac{t}{\log(t)^{\max\{1/p,1/p'\}}} \right).$$
	\end{thm}
	
	\noindent First, we need the following.
	
	\begin{prop}\label{propAlmpACB}
		Let \Tt \,be a uniformly eventually positive \CB $C_0$-semigroup on $L^p(\Omega)$. Then there exists $C>0$ such that for every $x \in L^p(\Omega)$ and every $t\geq 1$,
		\begin{equation}\label{propEgalAlmpACB}
			\int_{0}^{t}\norm{T_sx}_{L^p}^p ds \leq Ct^p \norm{x}_{L^p}^p.
		\end{equation}
	\end{prop}
	\begin{proof}
		Let $t_0$ be such that $T_t$ is positive for every $t\geq t_0$.
		We make slight changes to the proof of \cite[Theorem 2]{Vogt} to be able to use the Ces\`aro boundedness property. Let $t>2t_0+1$ and $x\in L^p(\Omega)_+$. Let $h\in L^{p'}([2t_0+1,t]; \mathbb{R})$ be nonnegative and such that $\|h\|_{p'} \leq 1$ and extend $h$ on $\mathbb{R}_+$ by setting $0$ outside of $[2t_0+1,t]$. We have
		\begin{align*}
			\int_{2t_0+1}^t h(s)T_sf ds
			& = \int_{2t_0+1}^t \left(\int_{s-t_0-1}^{s-t_0} du\right) h(s)T_sx ds \\
			& = \int_{t_0}^{t-t_0} \left(\int_{t_0+u}^{t_0+u+1} h(s)T_sx ds \right) du \\
			& = \int_{t_0}^{t-t_0} T_u \left( \int_{t_0}^{t_0+1} h(s+u)T_sx ds \right) du.
		\end{align*}
		Define $f:=\displaystyle \left( \int_{t_0}^{t_0+1} (T_sx)^p ds \right)^{1/p} \in L^p(\Omega)$. Let $M=\sup_{u\in [t_0,t_0+1]} \ \|T_u\|$. By Fubini's theorem, we have
		$$
		\|f\|^p = \int_{t_0}^{t_0+1} \|T_sx\|^p ds \leq M^p \|x\|^p.
		$$
		Now, note that, by H\"older's inequality,
		$$
		\int_{t_0}^{t_0+1} h(s+u)T_sx ds \leq f
		$$
		so that, by positivity of $(T_r)_{r\geq t_0}$, we have
		$$
		\int_{2t_0+1}^t h(s)T_sx ds \leq \int_{t_0}^{t-t_0} T_uf du.
		$$
		Taking the supremum over such $h$ and applying \cite[Lemma 1]{Vogt} yields
		$$
		\left(\int_{2t_0+1}^t (T_sx)^p ds\right)^{1/p} \leq \int_{t_0}^{t-t_0} T_uf du.
		$$
We integrate the previous inequality over $\Omega$ and use the Ces\`aro boundedness property to obtain
$$
\int_{2t_0+1}^t \|T_sx\|^p ds \leq \left\| \int_{t_0}^{t-t_0} T_uf du \right\|^p \leq C'^p t^p \|f\|^p \leq C'^p M^p t^p \|x\|^p,
$$
which concludes the proof.
	\end{proof}

\begin{remark} The previous proposition is in fact much easier for the special case $p=1$. Indeed, if $(T_t)_{t\geq 0}$ is a positive and Ces\`aro bounded $C_0$-semigroup on $L^1(\Omega, \mu)$, then, for a positive $x\in L^1(\Omega, \mu)$,
$$
\int_{0}^{t} \norm{T_sx}ds = \int_{0}^{t} \left\langle T_sx, 1 \right\rangle ds = \left\langle \int_{0}^{t} T_sx ds, 1 \right\rangle = \norm{\int_{0}^{t} T_sx ds} \leq Ct \|x\|.
$$
The same applies to a uniformly eventually positive semigroup by splitting up the integral into two integrals.
\end{remark}
	
	Next, we prove that the estimate obtained in the previous proposition allows us to obtain the desired estimate for $\|T_t\|$.
	
	\begin{lemma}\label{normelog}
		Let $1 < p < +\infty$. Let $(T_t)_{t \geq 0}$ be a $C_0$-semigroup on a reflexive Banach space $X$ and assume that there is a function $f : (1,+\infty) \to \mathbb{R_+}$ and a constant $C\geq 0$ such that, for any $t\geq 2$ and any $x\in X$ and $x^* \in X^*$,
		$$
		\left(\int_{0}^t \norm{T_s x}^p ds \right)^{1/p} \leq f(t) \norm{x}
		\ \ \ 
		\text{and} \ \ \ 
		\left(\int_{0}^t \norm{T^*_s x^*}^{p'} ds\right)^{1/p'} \leq Ct \norm{x^*}.
		$$
		Then, as $t\to +\infty$, $\norm{T_t}= \mathcal{O}\left( \dfrac{f(t)}{\log(t)^{1/p}} \right)$. 
	\end{lemma}
	
	\begin{proof}
		Let $t >2$ and $2\leq \alpha <  \beta \leq t$. For any $x\in X$ and $x^* \in X^*$, we have
		\begin{align*}
			(\beta-\alpha)|\langle T_tx,x^* \rangle | &=   \left|\int_{\alpha}^{\beta} \langle T_{t-s}x,T_s^*x^* \rangle ds \right|\\
			& \leq \left(\int_{\alpha}^{\beta} \norm{T_{t-s}x}^p ds\right)^{1/p} \left(\int_{\alpha}^{\beta} \norm{T_{s}^*x^*}^{p'}ds\right)^{1/p'} \\
			&  \leq  C \beta \|x^*\| \left(\int_{t-\beta}^{t-\alpha} \norm{T_u x}^p du\right)^{1/p}.
		\end{align*}
		After taking the supremum over $\{\norm{x^*} =1 \}$, we get
		\[
		\frac{(\beta-\alpha)^p}{\beta^p} \norm{T_tx}^p \leq C^p \int_{t-\beta}^{t-\alpha} \norm{T_u x}^p du.
		\]
		Now we set $L:= \left\lfloor \log(t)/log(2) \right\rfloor$. For $0 \leq l \leq L-1$, the latter inequality applied to $\beta=2^{l+1}$ and $\alpha  = 2^{l}$ reads
		\[
		\norm{T_tx}^p \leq 2^{p} C^p \int_{t-2^{l+1}}^{t-2^l} \norm{T_ux}^pdu.
		\]
		Summing this inequality over $l$ now gives
		\[
		L\norm{T_tx}^p = \sum_{l=0}^{L-1} \norm{T_tx}^p \leq \sum_{l=0}^{L-1} 2^{p} C^p \int_{t-2^{l+1}}^{t-2^l} \norm{T_ux}^p du \leq 2^{p} C^p \int_0^t \norm{T_ux}^p du.
		\]
		Hence, the assumption on $T_t$ yields, for $t>2$, 
		\[
		\norm{T_tx}^p \leq \frac{2^{p} C^p f(t)^p \|x\|^p}{L}
		\]
		so that
		\[
		\norm{T_t} \leq \frac{2 C f(t)}{L^{1/p}}  = \mathcal{O}\left( \dfrac{f(t)}{\log(t)^{1/p}} \right),
		\]
		which concludes the proof.
	\end{proof}
	
	\begin{remark}\label{rkestilogp}
		The same proof applies if we permute the role of $T_t$ and $T_t^*$. In particular, if for every $x\in X$ and $x^*\in X^*$,
		$$
		\int_{0}^t \norm{T_s x}^p ds \leq Ct^p \norm{x}^p \quad	\text{and} \quad	\int_{0}^t \norm{T^*_s x^*}^{p'} ds \leq Ct^{p'} \norm{x^*}^{p'},
		$$
		then $$\norm{T_t} = \mathcal{O}\left( \dfrac{t}{\log(t)^{\max\{1/p,1/p'\}}} \right).$$
	\end{remark}

	\noindent We can now prove Theorem \ref{maintheorem}.
	
\begin{proof}[Proof of Theorem \ref{maintheorem}]
Since \Tt \, is uniformly eventually positive and Kreiss bounded, \Tt \, is \CB and by Proposition \ref{propAlmpACB}, \Tt \, satisfies \eqref{propEgalAlmpACB}. Moreover $(T_t^*)_{t\ge0}$ is a uniformly eventually positive \KB $C_0$-semigroup on $L^{p'}(\Omega)$ so it is \CB and, again, by Proposition \ref{propAlmpACB}, $(T_t^*)_{t\ge0}$ satisfies \eqref{propEgalAlmpACB} with $p'$ instead of $p$. Finally, by Remark \ref{rkestilogp}, we get the desired result.
\end{proof}
	
\begin{remark}\label{alphaKreiss} One can define, more generally, the notion of $\alpha$-Kreiss bounded semigroups. Namely, we say that \Tt \, is $\alpha$-Kreiss bounded for some $\alpha \geq 0$ if, for every $\lambda \in \mathbb{C}_+$,
\[
		\norm{R(-\lambda,A)} \leq C\left(\frac{1}{\text{Re}(\lambda)^{\alpha}}+1\right).
\]
If $\alpha=1$, it is straightforward to check that this is equivalent to the definition of Kreiss boundedness given in Definition \ref{defNotions}.
Let \Tt \, be $\alpha$-Kreiss bounded and positive on $L^p$ for some $1\leq p<\infty$.

\begin{enumerate}
    \item If $0\leq \alpha < 1$ then, \cite[Page 90, Remark 1.22]{Eisnerbook} together with \cite[Theorem 1]{Weis98} show that $(T_t)_{t\ge 0}$ is in fact exponentially stable.
    \item If $\alpha>1$, then adapting the results of this paper (replacing $t$ by $t^{\alpha}$ in various places and examining the proof of Lemma \ref{normelog}) allow us to get	$$\norm{T_t}= \mathcal{O}\left( t^{\alpha} \right), \ t\geq 1,$$
    retrieving \cite[Theorem 1.1 (3)]{RozVer}. We leave the details to the reader.
\end{enumerate}
\end{remark}

	Let us now discuss the discrete analogue of Theorem \ref{maintheorem}. Let $T$ be a bounded operator on a Banach space $X$. We say that $T$ is Ces\`aro bounded if
	$$
	\sup_{n\geq 1} \dfrac{1}{n} \left\| \sum_{k=0}^{n-1} T^k \right\| < +\infty.
	$$
	Similarly, one can define, for $T$ (or for the discrete semigroup $(T^n)_{n\in \mathbb{N}}$), the notion of (absolute, uniform) Kreiss boundedness. We refer to \cite{Cuny2} for the relevant definitions. Following Section \ref{Section2} of this paper, one can prove that for an individually eventually positive operator $T$ on $L^p$, these notions are equivalent. Moreover, it is straightforward to adapt the results of Section \ref{Section3} so that we can get the following result.
	
\begin{thm}\label{maintheoremdiscrete}
	Let $1< p<\infty$ and let $T$ be a uniformly eventually positive Ces\`aro bounded 
	on $L^p(\Omega)$. Then,
		$$\norm{T^n}= \mathcal{O}\left( \dfrac{n}{\log(n)^{\max\{1/p,1/p'\}}} \right).$$
\end{thm}

	\begin{remark}
		This theorem improves \cite[Corollary 3.9]{Cuny2} where the growth rate that is obtained is $\norm{T^n}= \mathcal{O}\left( n/\sqrt{\log(n)} \right)$.
	\end{remark}

	\begin{q}
		\begin{enumerate}
		\item
In Theorem \ref{maintheorem}, we assumed that the semigroup is uniformly eventually positive and the same assumption was made in \cite{Vogt} to prove the equality between the growth bound and the spectral
bound of its generator. It is not known whether this equality holds if we assume that the semigroup is individually eventually positive. Hence, it is natural to ask if Theorem \ref{maintheorem} remains true if we replace "uniformly eventually positive" by "individually eventually positive.

		\item In view of Section \ref{Section4}, we can also ask whether the estimate in Theorem \ref{maintheorem} can be improved. For example, if $(T_t)_{t\ge 0}$ is a uniformly eventually positive Kreiss bounded $C_0$-semigroup on $L^p$, $1< p<+\infty$, do we have $\|T_t\| = \mathcal{O}\left( t^{1-\epsilon} \right)$ for some $\epsilon\in (0,1)$ ?
			\end{enumerate}
	\end{q}

\section{Growth rate of individually eventually positive Kreiss bounded $C_0$-semigroups on $(\text{AL})$-spaces and $(\text{AM})$-spaces}	 \label{Section4}

In this section, we estimate the growth rate of Kreiss bounded $C_0$-semigroups on $(\text{AL})$ and $(\text{AM})$-spaces. We will also make the assumption that the semigroups are individually eventually positive on $(\text{AL})$-spaces, and uniformly eventually positive on $(\text{AM})$-spaces.


Let $X$ be a (real of complex) Banach lattice. Recall that $X$ is an $(\text{AL})$-space if its norm satisfies the following additive property
\begin{equation}\label{additiveAL}
    	\forall x,y \in X_+, \ \norm{x+y} = \norm{x} + \norm{y}.
\end{equation}
	Typical examples of $(\text{AL})$-spaces are the $L^1$-spaces. In fact, it is known  that every $(\text{AL})$-space is isomorphic (as a Banach lattice) with a suitable $L^1(\mu)$, see \cite{Kakutani}.
	
Next, we say that $X$ is an $(\text{AM})$-space if its norm satisfies 
$$
\forall x,y \in X_+, \ \norm{\sup{(x,y)}} = \sup{(\norm{x},\norm{y})}.
$$
For example, if $K$ is compact, the space $\mathcal{C}(K)$, or any sublattice, is an $(\text{AM})$-space. As recalled in the introduction, the dual of an $(\text{AM})$-space is an $(\text{AL})$-space.

We first establish an estimate for the growth rate of certain $C_0$-semigroups on $(\text{AL})$-space, see Theorem \ref{maintheoremAL}. Then, by a duality argument, we obtain, in  Theorem \ref{maintheoremAM}, a similar estimate for the growth of certain semigroups on $(\text{AM})$-spaces.


We will assume that all the Banach lattices are complex but it is easy to obtain similar results for real Banach lattices, see Remark \ref{realcase}.\\

		\begin{thm}\label{maintheoremAL}
		Let $(T_t)_{t\ge 0}$ be an individually eventually positive $C_0$-semigroup on 
		an $(\text{AL})$-space $X$. Assume that $(T_t)_{t\ge 0}$ is Kreiss bounded.
		Then, there exists $\epsilon \in (0,1)$ such that, as $t\to +\infty$,
			$$
			\norm{T_t} = \mathcal{O}(t^{1-\epsilon}).
			$$
	\end{thm}
	
Let us start with the following lemma. It is slightly more general than we really need but that can be of independent interest.

		\begin{lemma}\label{norme2}
		Let $(T_t)_{t \geq 0}$ be a $C_0$-semigroup on a Banach space $X$. Assume that there is an increasing function $f : [1,+\infty) \to (0,+\infty)$ such that,
		\begin{equation}\label{Lemmahyp}
		  \forall x\in X, \ \exists K_x>0, \ \forall t\geq 1, \    \int_{0}^t \norm{T_s x} ds \leq K_xf(t) \norm{x}.  
		\end{equation}
		Define $F$ on $(2,+\infty)$ by $F(u) = \displaystyle \int_2^u \dfrac{ds}{f(s)}$. Then, there is a constant $C\geq 1$ such that $$\norm{T_t}= \mathcal{O}\left(f(t)e^{-F(t)/C}\right).$$
	\end{lemma}
	
	\begin{proof}
		Fix a unit vector $x\in X$ and let $n\in \mathbb{N} \cup \lbrace 0 \rbrace$. The continuity of $s \mapsto \|T_sx\|$ ensures the existence of $t_n \in [n,n+1]$ such that
		\[
		\|T_{t_n} x \| = \int_n^{n+1} \|T_sx\| ds.
		\]
		Let $1\leq j \leq n$. We have
		\begin{align}\label{Thm1ineg}
		\|T_{t_n}x\| \leq \|T_{t_n-t_{n-j}}\|\|T_{t_{n-j}}x\|.
		\end{align}
		and since $j-1 \leq t_n-t_{n-j} \leq j+1$, we have, by letting $M = \sup_{u\in [0,2]} \ \|T_u\|$, $$\|T_{t_n-t_{n-j}}\| = \|T_{(t_n-t_{n-j}) - (j-1)} T_{j-1}\| \leq M \|T_{j-1}\|.$$
		Hence, by \eqref{Thm1ineg},
		$$\|T_{n+1}x\| \leq M\|T_{t_n}x\| \leq M \|T_{j-1}\| \|T_{t_{n-j}}x\|.$$
		From now on, assume that for every $k, \|T_k\| > 0$, otherwise the result is trivial. Summing the latter inequality from $j=1$ to $n$ gives
		$$
		\|T_{n+1}x\| \sum_{i=0}^{n-1} \dfrac{1}{\|T_i\|} \leq M \sum_{k=0}^{n-1} \|T_{t_k}x\| = M\int_0^{n} \|T_sx\| ds \leq MK_xf(n)\|x\|.
		$$
		The uniform boundedness principle and the fact that
		$$
		\|T_{n+1}\| \sum_{i=1}^{n+1} \dfrac{1}{\|T_{i}\|} \leq \|T_{n+1}\| \sum_{i=1}^{n-1} \dfrac{1}{\|T_{i}\|} + 2M,
		$$
	    give the existence of a constant $C$ such that
		$$
		\|T_{n+1}\| \sum_{i=1}^{n+1} \dfrac{1}{\|T_{i}\|} \leq Cf(n).
		$$
		By (the proof of) \cite[Lemma 3.2]{Cuny1}, we get that $$\|T_{n+1}\| = \mathcal{O}\left(f(n)e^{-F(n+1)/C}\right).$$
		Finally, if $t \geq 2$, we let $n\in \mathbb{N}$ be such that $t\in [n,n+1]$ so that we have
		\[
		\|T_t\| = \|T_{t-n}T_{n}\| \leq M \|T_n\|=\mathcal{O}\left(f(t)e^{-F(t)/C}\right).
		\]
		This concludes the proof of the Lemma.
	\end{proof}

	\begin{proof}[Proof of Theorem \ref{maintheoremAL}]	
	 An individually eventually positive Kreiss bounded $C_0$-semigroup $(T_t)_{t\geq 0}$ on a $(\text{AL})$-space $X$ satisfies the assumption (\ref{Lemmahyp}) of Lemma \ref{norme2} with $f(s)=s$. Indeed, let $x\in X_+$ and let $t_x \geq 0$ be such that, for every $t\geq t_x$, $T_tx \in X_+$. Property (\ref{additiveAL}) implies that
	 	$$
	\int_{t_x}^{t} \norm{T_sx}ds = \norm{\int_{t_x}^{t} T_sx ds}.
	$$
Let $M_x = \sup_{u\in [0,t_x]} \ \|T_u\|$. By Proposition \ref{propequiv}, $(T_t)_{t\ge 0}$ is Ces\`aro bounded. Hence, there is a constant $C\geq 0$ such that $\norm{\displaystyle \int_{0}^t T_s x ds} \leq Ct\|x\|$. In particular, for every $t\geq t_x$,
$$\norm{\int_{t_0}^t T_s x ds} = \norm{\int_0^t T_s x ds - \int_0^{t_0} T_s x ds} \leq (Ct+M_xt_0)\|x\| \leq (C+M_x)t \|x\|.$$
Hence, for $t\geq t_x$,
\begin{align*}
    \int_{0}^{t} \norm{T_sx}ds = \int_{0}^{t_x} \norm{T_sx}ds + \norm{\int_{t_x}^{t} T_sx ds} \leq (C+2M_x)t\|x\|.
\end{align*}
If $1\leq t \leq t_x$, we simply have
$$
\int_{0}^{t} \norm{T_sx}ds \leq M_x t \|x\|.
$$
This shows (\ref{Lemmahyp}) for $x\in X_+$ and since $X_+$ spans $X$, we obtain (\ref{Lemmahyp}) for every $x\in X$.

Now, we can apply Lemma \ref{norme2} and, for this choice of function $f$ and using the same notation as in the Lemma, we have $F(u)=\log(u)-\log(2)$ so that
	$$\norm{T_t}= \mathcal{O}\left(f(t)e^{-F(t)/C}\right) = \mathcal{O}\left(t^{1-1/C}\right),$$
	which finishes the proof of the Theorem.
	\end{proof}

\begin{remark}
	Conversely, for every $\epsilon\in (0,1)$, it is easy to construct a weighted translation semigroup $(T_t)_{t\ge 0}$ on $L^1(\mathbb{R}_+)$ (see \cite[Corollary 2.3]{Mull1} for the discrete case) which is Kreiss bounded and such that $\|T_t\| = \mathcal{O}(t^{1-\epsilon})$.
\end{remark}

As in Section \ref{Section3}, we can prove the following discrete analogue (a similar result had been obtained in \cite[Theorem 2]{KornKosek}) of Theorem \ref{maintheoremAL}.

\begin{thm}\label{maintheoremALdis}
	Let $T$ be an individually eventually positive Kreiss bounded operator on an $(\text{AL})$-space.
    Then, there exists $\epsilon \in (0,1)$ such that
	$$
	\norm{T^n} = \mathcal{O}(n^{1-\epsilon}).
	$$
\end{thm}

\begin{remark}\label{realcase}
It is clear from the proof of Theorem \ref{maintheoremAL} that if $X$ is a real $(\text{AL})$-space and $(T_t)_{t\ge 0}$ an individually eventually positive and Ces\`aro bounded $C_0$-semigroup on $X$, the conclusion that $
			\norm{T_t} = \mathcal{O}(t^{1-\epsilon})
			$
			remains true. The same remark applies to the discrete case in Theorem \ref{maintheoremALdis}.
\end{remark}


We can now prove a similar estimate for the growth rate of $C_0$-semigroups on $(\text{AM})$-spaces. We thank Johen Glück for his precious help in devising this proof.
		
\begin{thm}\label{maintheoremAM}
	Let $X$ be an $(\text{AM})$-space and let $(T_t)_{t\ge 0}$ be a uniformly eventually positive and Kreiss bounded $C_0$-semigroup on $X$.
	Then, there exists $\epsilon \in (0,1)$ such that, as $t\to +\infty$,
	$$
	\norm{T_t} = \mathcal{O}(t^{1-\epsilon}).
	$$
\end{thm}

	\begin{proof}
	Let $T:= T_1$. First, we show that $T$ is Ces\`aro bounded (as a bounded operator), that is 
	$$
	\exists C>0, \ \forall x \in X, \ \forall N\in \N, \ \norm{\sum_{k=1}^N T^kx} \le CN \|x\|.
	$$
	 	Let $t_0 \geq 0$ be such that for every $t\geq t_0, T_t$ is a positive operator. Let $x \in X_+$. By \cite[Proposition II.7.6]{Schaefer}, every nonempty relatively compact subset of $X$ has a supremum in $X$, so
	 	 the compact set $\{T_tx, t\in [t_0;t_0+1] \} \subset X_+$ has a supremum $y\in X_+$. 
	 Let $N_0\in \N$ be such that $N_0 \geq t_0 +1 $ and $k\in \N$ be such that $ k \geq N_0$. Notice that for every $s\in [k-t_0-1, k - t_0 ]$, we have 
	 $$
	 0 \leq T_kx = T_sT_{k-s}x \leq T_sy,
	 $$
	 from which we deduce that
	 $$
	  0 \leq T^kx = T_kx = \int_{k-t_0-1}^{k-t_0}T_kx ds \leq \int_{k-t_0-1}^{k-t_0} T_sy ds.
	 $$
	 Let $N \ge N_0 $. We sum the last inequality over $k=N_0,\ldots, N$ to obtain
	 $$
	 0\leq \sum_{k= N_0}^{N} T^kx \leq  \sum_{k= N_0}^{N} \int_{k-t_0-1}^{k - t_0 } T_sy ds = \int_{N_0-t_0-1}^{N - t_0 } T_sy ds.
	 $$
	  Hence, by Cesàro boundedness of $(T_t)_{t\ge0}$ and similar computations as in the proof of Theorem \ref{maintheoremAL}, there exists a constant $K$ such that
	  	$$
	  \norm{\sum_{k=N_0}^N T^kx} \leq \norm{\int_{N_0-t_0-1}^{N-t_0}T_{s}yds} \le KN\norm{y}.
	  $$ 
	  From this inequality, we easily deduce that
	  \begin{align*}
	  \sup_{N\geq 1} \dfrac{1}{N}  \norm{\sum_{k=1}^N T^kx} < +\infty.
	  \end{align*}
	  The same holds true for a general $x\in X$ so, by the uniform boundedness principle, we conclude that $T$ is Ces\`aro bounded.

Now, note that $T^*$ is also Ces\`aro bounded and uniformly eventually positive on $X^*$ which is an $(\text{AL})$-space. Hence, by Theorem \ref{maintheoremALdis}, $\norm{T^n} = \norm{(T^*)^n} = \mathcal{O}(n^{1-\epsilon})$. Since $T=T_1$, it is now straightforward to see that we have the desired growth rate for $T_t$.

	  \end{proof}


	\noindent \textbf{Acknowledgments.} The first author was supported by the ERC
	grant {\it Rigidity of groups and higher index theory} under the European Union’s
	Horizon 2020 research and innovation program (grant agreement no. 677120-INDEX).
	
	The authors wish to thank Christophe Cuny for valuable comments, suggestions and discussions at early stages of the preparation of this paper.
 
 The authors are indebted to an anonymous reviewer for his careful reading of the paper, his valuable comments and for having brought to our attention several references which allowed to improve the paper.


\end{document}